 \documentclass[dvips]{amsart}

\usepackage{amssymb, amsfonts}

 \usepackage{pb-diagram,lamsarrow,pb-lams}

\usepackage{url}
\usepackage[mathscr]{eucal} 
  \let\ms=\mathscr \let\mc=\mathcal \let\go=\mathfrak
  \let\bb=\mathbb \let\mb=\mathbf
\usepackage{hyperref}

\let\:=\colon
\let\?=\overline

\let\dl=\delta
\let\dg=\dagger
\let\into=\hookrightarrow
\let\onto=\twoheadrightarrow
\let\ox=\otimes
\let\ts=\textstyle
\let\vf=\varphi
 
\let\x=\times

\def\tqn#1{\text{\quad#1\enspace}}
\def\(#1){{\rm(#1)}}

\def\risom{\buildrel\sim\over{\smashedlongrightarrow}}
  \def\smashedlongrightarrow{\setbox0=\hbox{$\longrightarrow$}\ht0=1pt\box0}

\def\Im{\mathop{\rm Im}}
\def\Ker{\mathop{\rm Ker}}
\def\length{\mathop{\rm length}}
\def\Proj{\mathop{\rm Proj}}

\def\rank{{\mathop{\rm rank}}}
\def\Spec{\mathop{\rm Spec}}
\def\Supp{\mathop{\rm Supp}}
\def\Sym{\mathop{\rm Sym}}

\DeclareMathOperator\Fitt{{\rm Fitt}}

\newtheorem{lemma}{Lemma}
\newtheorem{proposition}[lemma]{Proposition}

\newtheorem{theorem}[lemma]{Theorem}

\theoremstyle{definition}
 \newtheorem{remark}[lemma]{Remark}
 \newtheorem{definition}[lemma]{Definition}
 \newtheorem{example}[lemma]{Example}

\numberwithin{equation}{lemma}

\begin{document}

\begin{figure}[t]
\raggedleft \slshape
To the memory of Alexandru (Bobi) Lascu:\\
a dear friend, tireless colleague, and inspiring coauthor.
\end{figure}

\title[Two Formulas]
{Two Formulas for the BR Multiplicity}


\author[S. L. Kleiman]{Steven L. Kleiman}
 \address
 {Dept.\ of Math., 2-172 MIT\\
 77 Mass.\ Ave.\\
 Cambridge, MA 02139, USA}
 \email{Kleiman@math.MIT.edu}



\date{\today}

\begin{abstract}
 We prove a {\it projection formula,} expressing a relative
Buchsbaum--Rim multiplicity in terms of corresponding ones over a
module-finite algebra of pure degree, generalizing an old formula for
the ordinary (Samuel) multiplicity.  Our proof is simple in spirit:
after the multiplicities are expressed as sums of intersection numbers,
the desired formula results from two projection formulas, one for cycles
and another for Chern classes.  Similarly, but without using any
projection formula, we prove an expansion formula, generalizing the {\it
additivity formula} for the ordinary multiplicity, a case of the {\it
associativity formula}.
 \end{abstract}

\maketitle


We prove a projection formula \eqref{eqMain} and an expansion formula
\eqref{eqAF} for the (relative) Buchsbaum--Rim multiplicity.  A special
case of \eqref{eqMain} was requested by Gaffney and Rangachev for use in
an example at the end of Sct.\,4 in their paper \cite{GR}.  We prove
\eqref{eqAF} similarly, but more simply.

Gaffney and Rangachev work in the following setting.  Let $R$ be the
local ring of a 1-dimensional reduced and irreducible analytic germ, $F$
a free module of finite rank, $M\subset N$ a generically equal nested
pair of submodules.  Let $R'$ be the normalization of $R$.  Set
$N':=N\ox R$ and $F_1:=F\ox R'$.  Form the $R'$-submodule $M'$ of $N'$
generated by $M$; form the $R'$-submodules $M_1$ and $N_1$ of $F_1$
generated by $M$ and $N$.  Consider the multiplicities $e(M,N)$ and
$e(M',N')$ and $e(M_1,N_1)$.

Formula \eqref{eqMain} yields $e(M,N)=e(M',N')$.  But
$e(M',N')=e(M_1,N_1)$; see the last paragraph of Rmk.\,\ref{reSthy}.
Now, $R'$ is a discrete valuation ring; so $M_1$ and $N_1$ are free of
the same rank.  Hence, as $R'$ is Cohen--Macaulay, $e(M_1,N_1)$ is, by
Buchsbaum and Rim's Cor.\,4.5 in \cite{BR}, equal to the length
$\ell(N_1/M_1)$, and so equal to the length of the zeroth Fitting ideal
$\ell(\Fitt^0(N_1/M_1)$.  Thus $e(M,N)=\ell(\Fitt^0(N_1/M_))$.  The
latter length is computed with ease at the end of Sct.\,4 in \cite{GR}.

Formula \eqref{eqMain} is our main result.  It is a general projection
formula, expressing a (relative) Buchsbaum--Rim multiplicity over a
Noetherian local ring of any positive dimension in terms of
corresponding multiplicities over a module-finite algebra of pure
degree.  The latter notion is defined in Dfn.\,\ref{dePd}, and
illustrated in Ex.\,\ref{exPd}.

Our proof is simple in spirit.  The multiplicities are expressed as sums
of intersection numbers following \cite{GTBR}.  Then \eqref{eqMain}
results from a projection formula for cycles \eqref{eqlePF} combined with
a projection formula for Chern classes.  The devil is in the details!

We derive \eqref{eqlePF} from a projection formula in Fulton's book
\cite{Fu}.  To do so, we compute the direct image of the top part of a
fundamental cycle; our computation is lengthy, but elementary.  Although
Fulton assumed his schemes are of finite type over a field (or more
generally, any local Artinian ring\,---\,see fn.\,1 on p.\,6), his proof
works here without change, essentially because on a reduced and
irreducible Noetherian scheme, a Cartier divisor is still of pure
codimension 1.

Formula \eqref{eqlePF} equates two cycles with the same divisor as
support.  In the proof of \eqref{eqMain}, all the divisors in play are
proper over (fat) points.  So the requisite theory of Chern classes and
intersection numbers is covered by the first three chapters of Fulton's
book \cite{Fu}.  A simpler, adequate alternative is developed in App.\,B
of \cite{Pic}.

In \cite{GTBR}, the Buchsbaum--Rim multiplicity is defined for a pair of
standard graded algebras.  However, we have an intrinsic choice of such
an algebra associated to any finitely generated module $N$, namely its
Rees algebra $\mc R(N)$, which was introduced and studied by Eisenbud,
Huneke, and Ulrich in \cite{EHU}.  Their definition is recalled in
Dfn.\,\ref{desRM}, so that we can use it there to define the
Buchsbaum--Rim multiplicity $e(M,N)$ for a suitable nested pair of
modules $M\subset N$.

In Complex Analytic Singularity Theory, the nested modules arise as
submodules of a free module.  So it is natural to consider the
subalgebras generated by the submodules inside the symmetric algebra of
the free module.  Usually the singularity is reduced, and then these
subalgebras are equal to the Rees algebras.  Moreover, even if the
singularity is only generically reduced, then these subalgebras can be
used to compute the Buchsbaum--Rim multiplicity.  We discuss this matter
in Rmk.\,\ref{reSthy}.

Prp.\,\ref{prAF} provides the expansion formula \eqref{eqAF}.
Prp.\,\ref{prAF} recovers via a new proof a version of Kirby and Rees's
Thm.\,6.3,\,iii) in \cite{KR}.  Moreover, \eqref{eqAF} recovers the {\it
associativity formula} for the ordinary multiplicity found on p.\,284 of
Eisenbud's book \cite{Eis} and in Thm.\,11.2.4 with $M:=R$ on p.\,218 of
Huneke and Swanson's book \cite{HS}.  However, \eqref{eqAF} does not
recover the older, more general associativity formula discussed in
Rmk.\,\ref{reZS}.  Our proof starts like our proof of
Thm.\,\ref{thMain}, but is simpler.

Thm.\,\ref{thMain} generalizes a similar result for ordinary
multiplicities, which was proved by Zariski and Samuel in \cite{ZS} via
the theory of Hilbert--Samuel polynomials.  We explain this matter and
more in Rmk.\,\ref{reZS}, which closes this paper.

\begin{definition}\label{dePd}
 Let's say a finite map of schemes $\pi\:B'\to B$ is of {\it pure
degree} $\dl$ if $\dl>0$ and if there's a dense open set $U$ of $B$ with
$\pi^{-1}U$ dense in $B'$ and with $(\pi_*\ms O_{B'})|U$ a free $\ms
O_{U}$-module of pure rank $\dl$.

Let's say a module-finite $R$-algebra $R'$ is of {\it pure degree} $\dl$
if the corresponding map of schemes $\Spec R'\to \Spec R$ is of pure
degree $\dl$.  It is equivalent, by prime avoidance, that there be an
$f\in R$ belonging to no minimal prime of $R$ and to no contraction of a
minimal prime of $R'$ with $R'_f$ a free $R_f$-module of pure rank
$\dl$.  \end{definition}

\begin{example}\label{exPd}
Fix  a finite, surjective map of schemes $\pi\:B'\to B$.

First, if $\pi$ is birational, then there's a dense open set $U$ of $B$
with $\pi^{-1}U$ dense in $B'$ and with $\pi|\pi^{-1}U$ an isomorphism
$\pi^{-1}U\risom U$.  Thus $\pi$ is of pure degree 1.

Second, assume $B$ is reduced.  Then, as is well known, there's a dense
open set $U$ with $(\pi_*\ms O_{B'})|U$ a free $\ms O_{U}$-module; for
example, see (*) on p.\,56 of Mumford's book \cite{Mum}.  However, the
rank $\delta$ of $(\pi_*\ms O_{B'})|U$ may vary from component to
component; so assume $B$ is irreducible.  Also, assume that every
component of $B'$ maps onto $B$, so that $\pi^{-1}U$ is dense.  Then
$\pi$ is of pure degree $\delta$.

 Third, assume $\pi$ is of pure degree $\dl$.  Given $x\in U$, let $K$
be an algebraically closed field containing the residue field of $x$.
Then $(\pi_*\ms O_{B'})_x$ is a free $\ms O_{B,x}$-module of rank $\dl$,
so $(\pi_*\ms O_{B'})\ox K$ is a $K$-vector space of dimension $\dl$.  On
the other hand, $(\pi_*\ms O_{B'})\ox K$ is the product of the local
rings of the ``geometric'' fiber $\pi^{-1}\?x$ where $\?x:=\Spec(K)$.
Assume also that $x$ lies outside of the discriminant locus.  Then these
local rings are each a copy of $K$.  Thus the cardinality of
$\pi^{-1}\?x$ is just $\dl$.
 \end{example}

\begin{lemma}\label{lePF}
 Let $\pi\:B'\to B$ be a finite map of pure degree $\dl$ of Noetherian
schemes of positive dimensions $s',\,s$, and $D$ a (Cartier) divisor on
$B$ such that $\pi^*D$ is a well-defined (Cartier) divisor on $B'$.
Then $\pi$ is surjective, and maps each component of $B'$ onto a
component of $B$.  Moreover, $s'=s$ and
 \begin{equation}\label{eqlePF}
 \pi_*\bigl(\pi^*D\cdot [B']_s\bigr)=\dl\,D\cdot[B]_s.
 \end{equation}
where $[B]_s$, $[B']_s$ are the $s$-dimensional parts of the fundamental
cycles of $B$, $B'$.
  \end{lemma}

\begin{proof}
 Since $\pi$ is of pure degree $\dl$, there's a dense open set $U$ of $B$
with $\pi^{-1}U$ dense in $B'$ and with $(\pi_*\ms O_{B'})|U$ a free
$\ms O_{U}$-module of pure positive rank.  Then $\pi$ restricts to a
surjection $\pi^{-1}U\onto U$.  So $U\subset \pi B'$.  But $\pi B'$ is
closed as $\pi$ is finite, and $U$ is dense in $B$.  So $\pi B'=B$.
Thus $\pi$ is surjective.  Thus, as $\pi$ is finite, $s'=s$.

Given a component $B'_1$ of $B'$, let $V$ be the complement in $B'_1$ of
the union of the other components.  Set $W:=V\cap \pi^{-1}U$.  Then $W$
is a nonempty open subset of $B'$, as $\pi^{-1}U$ is open and dense in
$B'$.  So $\pi W$ is nonempty; also, as $\pi^{-1}U\onto U$ is flat, $\pi
W$ is open in $U$, so in $B$.  But $\pi W\subset \pi B'_1$.  Moreover,
$\pi B'_1$ is closed, as $\pi$ is finite.  Thus $\pi B'_1$ is a
component of $B$.

As to Eqn.\,\eqref{eqlePF}, notice that Prp.\,2.3.\,(c) on p.\,34 of
Fulton's book \cite{Fu} yields $\pi_*(\pi^*D\cdot[B']_s)=
D\cdot\pi_*[B']_s$.  Thus it remains to prove $\pi_*([B']_s) =
\dl\,[B]_s$.

Recall that $[B']_s:= \sum_in_i[B'_i]$ where the $B'_i$ are all the
$s$-dimensional (irreducible) components of $B'$ with their reduced
structure and where $n_i$ is the length of the local ring $\ms
O_{B',\zeta_i}$ at the generic point $\zeta_i$ of $B'_i$.  So
$\pi_*[B']_s= \sum_in_i\pi_*[B'_i]$.  Also $\pi_*[B'_i]:=
\dl_i[\pi(B'_i)]$ where $\dl_i$ is the degree of the extension of the
function fields of $B_i$ over $\pi B'_i$.  Note that $\dl_i$ is finite,
that $\pi B'_i$ is closed, and that $\dim\pi B'_i= \dim B'_i$, all
because $\pi$ is finite.  But $\dim B'_i =s$ and $\dim B =s$.  Thus $\pi
B'_i$ is an $s$-dimensional component of $B$.

First, assume $B$ is reduced and irreducible.  Then $\pi_*[B']_s=
\sum_in_i\dl_i[B]$ by the above.  Let $\eta$ be the generic point of
$B$.  Then $\ms O_{B,\eta}$ is a field.  But $\pi$ is of pure degree
$\dl$.  So $(\pi_*\ms O_{B'})_\eta$ is a $\ms O_{B,\eta}$-vector space
of dimension $\dl$.  Thus $(\pi_*\ms O_{B'})_\eta$ is the product of all
the local rings $\ms O_{B',\zeta}$ for $\zeta\in \pi^{-1} \eta$;
also, $\sum_\zeta\dim_{\ms O_{B,\eta}}\ms O_{B',\zeta}=\dl$.

Given a $\zeta$, let $B'_\zeta$ be its closure.  Then $\pi B'_\zeta$ is
closed and contains $\eta$; so $\pi B'_\zeta= B$.  So $\dim B'_\zeta=
\dim B$.  But $\dim B =s$ and $\dim B'=s$.  Hence $B'_\zeta$ is an
$s$-dimensional component of $B'$.  Thus $\zeta=\zeta_i$ for some $i$.
Moreover, $n_i$ is the length of $\ms O_{B',\zeta_i}$ over itself, and
$\dl_i$ is the dimension of its residue field as a $\ms
O_{B,\eta}$-vector space.  Hence $\sum_in_i\dl_i=\dl$.  Thus
$\pi_*[B']_s=\dl\,[B]$ when $B$ is reduced and irreducible.

In general, $[B]_s=\sum m_j[B_j]$ where the $B_j$ are the
$s$-dimensional components of $B$ with their reduced structure and where
$m_j$ is the length of the local ring $\ms O_{B,\eta_j}$ at the
generic point $\eta_j$ of $B_j$.  Let's now prove $[B']_s= \sum
m_j[\pi^{-1}B_j]_s$.

Given $j$ and $\zeta\in \pi^{-1}\eta_j$, there's $i$ with
$\zeta_i=\zeta$ by the above.  But $\pi$ is of pure degree $\dl$.  So
$\ms O_{B',\zeta_i}$ is flat over $\ms O_{B,\eta_j}$.  Hence
$n_i=m_jl_i$ where $l_i$ is the length of $\ms O_{\pi^{-1}B_j,\zeta_i}$
over itself by Lem.\,A.4.1 on p.\,413 in \cite{Fu} (cf.\ Lem.\,1.7.1
on p.\,18).  Thus $[B'_i]$ appears with the same coefficient in both
$[B']_s$ and $\sum m_j [\pi^{-1}B_j]_s$; moreover, $[B'_i]$ is an
arbitrary component of this sum.  Conversely, if $[B'_i]$ is an
arbitrary component of $[B']_s$, then $\pi B'_i=B_j$ for some $j$ by
the above.  Thus $[B']_s= \sum m_j[\pi^{-1}B_j]_s$.

Finally, $\pi^{-1}B_j\to B_j$ is, for all $j$, plainly finite and of
pure degree $\dl$.  So by the first case, $\pi_*[\pi^{-1}B_j]_s=
\dl\,[B_j]$. Thus $\pi_*[B']_s=\sum m_j\dl\,[B_j]=\dl\,[B]_s$, as desired.
  \end{proof}

\begin{definition}\label{desRM}
 Let $R$ be a Noetherian ring, and $N$ a finitely generated module.
Following \cite[Dfn.\,0.1]{EHU}, define the {\it Rees algebra} $\mc
R(N)$ this way:
 $$\mc R(N):=\Sym(N)/L \tqn{where}
 L:=\bigcap\,\{\,\Ker(\Sym(u))\mid u\:N\to F \text{ with $F$ free}\,\}.$$

For each minimal prime $\go p$ of $R$, assume $N_\go p$ is a free $R_\go
p$-module of positive rank $r_\go p$, and set $d_\go p:=\dim(R/\go p)$.
Define the invariant $s(N)$ by the following formula:
 $$\ts s(N):= \max_\go p (d_\go p+r_\go p-1).$$

Finally, assume $R$ is local too, and let $M\subset N$ be a submodule.
Set $s:=s(N)$, set $P:= \Proj(\mc R(N))$, and set $\mb S:=[P]_s$.
Define the (relative) {\it Buchsbaum--Rim multiplicity\/} $e(M,\,N)$ to
be the number $e(\mb S)$ introduced in Sec.\,(5.1) in \cite{GTBR} with
$G':=\mc R(M)$ and $G:=\mc R(N)$; for details, see the beginning of the
proof of Thm.\,\ref{thMain}.
 \end{definition}

\begin{theorem}\label{thMain}
 Let $R$ be a Noetherian local ring of positive dimension, $M\subset N$
nested finitely generated modules, $R'$ a module-finite algebra of pure
degree $\dl$.  Assume $N$ is generically free of positive rank, and
$N/M$ is of finite length.  For every maximal ideal $\go m$ of $R'$, let
$\dl_\go m$ be the degree of the residue field extension of $R'_\go m$
over $R$, set $N'_\go m:=N\ox R'_\go m$, and let $M'_\go m\subset N'_\go
m$ be the $R'_\go m$-submodule generated by $M$.  Set $s:=s(N)$ and
$s_\go m:=s(N'_\go m)$ for all $\go m$.  Let $\Phi$ be the set of\/ $\go
m$ with $s_\go m=s$.  Then the Buchsbaum--Rim multiplicities satisfy
this relation:
 \begin{equation}\label{eqMain}
 \ts \dl\,e(M,\,N)=\sum_{\go m\in\Phi}\dl_\go m\,e(M'_\go m,\,N'_\go m).
 \end{equation}
\end{theorem}

\begin{proof}
 Let's recall the details of the definition of $e(M,\,N)$.
Figure~\ref{f1g1} shows the schemes involved and the canonical maps
relating them (just the top square is Cartesian).%
 \begin{figure}[!ht]
\dgARROWLENGTH=0.5\dgARROWLENGTH
\begin{minipage}[c]{0.35\textwidth}
\[\begin{diagram}
\node[3]{D}\arrow[3]{e}\arrow{sww,J}\node[3]{Z}\arrow{sww,J}\arrow[2]{s}\\
\node{B}\arrow[2]{s}\arrow[3]{e}\node[3]{P}\arrow[2]{s}\\ 
\node[6]{Y\rlap{$\ni x$}}\arrow{sww,L}\\
\node{Q}\arrow[3]{e}\node[3]{X}
\end{diagram}\]
\end{minipage}
\begin{minipage}[c]{0.6\textwidth}
\begin{align*}
B:=&\text{Bl}(Z,\,P)&
   P:=&\Proj(\mc R(N))\\
Q:=&\Proj(\mc R(M))&
   X:=&\Spec(R)\\
D:=&\text{Exceptional divisor}&
   Z:=&\bb V(M\mc R(N))\\
x:=&\text{Central point}&
   Y:=&\Supp(N/M)
\end{align*}
\end{minipage}%
\caption{The schemes involved in the definition of $e(M,\,N)$}\label{f1g1}
\end{figure}
  Here, $Z$ is the subscheme defined by the ideal $M\mc R(N)$, and $B$
is the blow-up of $P$ along $Z$.  Note that $D$ is proper over $Y$ and
that $Y$ is supported on the central point (that is, the closed point)
$x\in X$.

Form the intersection class $\Sigma$ on $D$ defined as follows:
 $$\ts \Sigma:=\sum_{i=1}^s\mu^{i-1}\nu^{s-i}\cap(D\cdot[B]_s)
 \text{\quad where\enspace}  \mu:=c_1\ms O_P(1)
 \text{\enspace and\enspace} \nu:=c_1\ms O_Q(1).$$
Then $e(M,\,N)$ is defined as the corresponding intersection number:
\begin{equation}\label{eqDfnBR}\ts
 e(M,\,N) := \int_{D/Y}\Sigma.
\end{equation}
 Thus $e(M,\,N)$ is the coefficient of $[x]$ in the projection of
$\Sigma$ from $D$ to $Y$.

 Let's identify the components of $P$ and $B$.  Given a component $X_1$
of $X$, let $\eta$ be its generic point, and $K$ the residue field of
$\ms O_\eta$.  Then the fiber $P_\eta$ is just $\Proj(\mc R(N)\ox_R K)$.
But $\mc R(N)\ox_R K=\mc R(N)\ox_R\ms O_\eta\ox_{\ms O_\eta}K$.  Set
$N_\eta:=N\ox_R\ms O_\eta$.  By \cite[Prp.\,1.3]{EHU}, forming a Rees
algebra commutes with flat base change; hence, $\mc R(N)\ox_R\ms O_\eta
=\mc R(N_\eta)$.  But $N_\eta$ is a free $O_\eta$-module; so plainly
$\mc R(N_\eta)= \Sym(N_\eta)$.  But $\Sym(N_\eta)\ox_{O_\eta} K =
\Sym(N_\eta\ox_{O_\eta}K)$.  Also $N_\eta\ox_{O_\eta}K=N\ox_RK$.
Therefore, $\mc R(N)\ox_R K=\Sym(N\ox_R K)$.  Set $r_1:=\rank(N_\eta)$,
so $r_1>0$.  Thus $P_\eta=\bb P_K^{r_1-1}$.

So $P_\eta$ is irreducible and nonempty. So $P_\eta$ lies in a component
$P_1$ of $P$.  Note $P_1$ projects onto a closed, irreducible subset $W$
containing $\eta$; so $W=X_1$.  Moreover, owing to
\cite[Lem.\,(3.1)]{GTBR}, given any proper map $S\to T$ between
irreducible Noetherian schemes, if the generic fiber is of dimension
$u$, then $\dim S =\dim T+u$.  Thus $\dim P_1 = \dim X_1 +r_1-1$.

Furthermore, $P_1$ corresponds to a minimal prime of $\mc R(N)$.  But
contraction sets up a bijection from the minimal primes of $\mc R(N)$ to
those of $R$ by \cite[Prp.\,1.5]{EHU}.  Thus projection sets up a
bijection from the components $P_i$ of $P$ to those $X_i$ of $X$.  Also
$\dim P := \max_i(\dim P_i) = s$.

The blow-up map $B\to P$ restricts to an isomorphism $B-D\risom P-Z$,
and $D$ is nowhere dense in $B$.  But $Z$ maps into $Y$, and $Y$ is
supported on $x\in X$.  Also each component $P_i$ of $P$ maps onto a
component of $X$; so $P_i$ does not map into $Y$, as $Y$ cannot be a
component of $X$ since $R$ is a local ring of positive dimension.  Hence
$Z$ is nowhere dense in $P$.  Thus projection sets up a bijection from
the components $B_i$ of $B$ to those $P_i$ of $P$, and each map $B_i\to
P_i$ is proper and birational.

So \cite[Lem.\,(3.1)]{GTBR} yields $\dim B_i=\dim P_i$.  Thus $\dim
B=s$.

For each maximal ideal $\go m$ of $R'$, an analogous setup arises from
the local ring $R'_\go m$ and its nested modules $M'_\go m\subset N'_\go
m$.  Denote the corresponding objects with a subscript of $\go m$.
Below, we check that $N'_\go m$ is generically free of positive rank.
Thus $e(M'_\go m,\,N'_\go m)$ is the coefficient of $[x_\go m]$ in the
projection of $\Sigma_\go m$ from $D_\go m$ to $Y_\go m$.  Moreover,
projecting $[x_\go m]$ to $Y$ yields $\dl_\go m[x]$.  Thus $\int_{D_\go
m/Y}\Sigma_\go m=\dl_\go m\,e(M'_\go m,\,N'_\go m)$.

Another analogous setup arises from the semilocal ring $R'$ and its
nested modules $M'\subset N'$ where $N':=N\ox R'$ and where $M'$ is
generated by $M$.  Denote the corresponding objects with a prime
($\prime$).  Below, we check that $N'$ is generically free of positive
rank.  Since forming a Rees algebra, forming a blowup, and so on commute
with flat base change, localizing this setup at $\go m$ yields the
preceding setup.

As above, $s_\go m=\dim P_\go m$ and $s'=\dim P'$.  But dimensions are
measured at closed points.  Moreover, each closed point of $P'_\go m$ is
a closed point of $P'$, and conversely each closed point of $P'$ is a
closed point of $P'_\go m$ for some $\go m$.  Thus $s_\go m\le s'$ for
each $\go m$, and $s_\go m=s'$ for some $\go m$.

Below, we check that $s'=s$.  Thus $\int_{D'/Y}\Sigma'= \sum_{\go
m\in\Phi}\dl_\go m\,e(M'_\go m,\,N'_\go m)$.

Let's now check that $N'$ and each $N'_\go m$ are generically free of
positive rank and that $s'=s$.  As $R'$ is of pure degree $\dl$, there's
$f\in R$ lying in no minimal prime of $R$ and in no contraction $\go p$
of a minimal prime $\go p'$ of $R'$; also $R'_f$ is a free $R_f$-module
of pure positive rank.  Given any $\go p'$, note that $\go p'R'_f$ is a
minimal prime.  But $R'_f$ is $R$-flat, so has the Going-down Property.
Hence $\go p\subset R$ is minimal.  So $N_\go p$ is free of positive
rank $r_\go p$.  But $N':=N\ox R'$. Thus $N'_{\go p'}$ is free of rank
$r_{\go p'}$ equal to $r_\go p$.

Each $\go m$ contains one or more $\go p'$.  Thus $N'_\go m$
is generically free of positive rank.

Finally, $R'/\go p'$ is a module-finite overdomain of $R/\go p$; so
$d_{\go p'}:=\dim (R'/\go p')$ is equal to $d_\go p:=\dim(R/\go p)$.
But every minimal prime $\go p$ of $R$ is the contraction of some $\go
p'$, because $\go pR_f$ is a minimal prime and $R'_f$ is a free
$R_f$-module of pure positive rank.  Since $d_{\go p'}=d_\go p$ and
$r_{\go p'} = r_\go p$, it follows that $s'=s$.

Yet another analogous setup arises from replacing $\mc R(N')$ by any
homogeneous quotient $\mc R^\dg$ with $\smash{\mc R(N')_f\risom\mc
R^\dg_f}$.  Denote the corresponding objects with a dagger ($\dg$).
There's no $N^\dg$ and no $Y^\dg$, but they're not needed.  Take
$M^\dg:=M'$ and $Z^\dg:=\bb V(M\mc R^\dg)$ and $s^\dg:=\dim P^\dg$.
Then $P^\dg$, $B^\dg$, and so on are closed subschemes of $P'$, $B'$,
and so on; in fact, $Z^\dg=Z'\cap P^\dg$ and $D^\dg=D'\cap B^\dg$.
Also, $\mu^\dg$ and $\nu^\dg$ are the pullbacks of $\mu'$ and $\nu'$.

Moreover, contraction sets up a bijection from the minimal primes of
$\mc R(N')$ to those of $R'$ by \cite[Prp.\,1.5]{EHU}.  Hence the
principal open set $P'_f\subset P'$ is (topologically) dense.  But
$P'_f=P^\dg_f\subset P^\dg\subset P'$.  So $\smash{P^\dg_f\subset
P^\dg}$ is dense too.  So $P'$ and $P^\dg$ have the same reductions.
Thus $s^\dg=s'$. But $s'=s$.  Thus $s^\dg=s'=s$.

The principal open subsets $B'_f\subset B'$ and $B^\dg_f\subset B^\dg$
are the preimages of $P'_f$ and $\smash{P^\dg_f}$.  But $B'\to P'$ and
$B^\dg\to P^\dg$ are birational.  Hence $B'_f$ is dense in $B'$, and
$B^\dg_f $ is dense in $B^\dg$.  Moreover, as schemes, $P'_f=P^\dg_f$
and $Z'_f=Z^\dg_f$, so $B'_f=B^\dg_f$.  Hence $[B']_s=[B^\dg]_s$.  Thus
$\Sigma'=\Sigma^\dg$ as classes on $D'$.

Choose $\mc R^\dg$ as follows.  Take any map $u\:N\to F$ with $F$ free
of finite rank such that the dual map $u^*\:F^*\to N^*$ is surjective;
for example, take $u$ to be the composition of the canonical map $N\to
N^{**}$ with the dual of any surjection from a free module of finite
rank onto $N^*$.  Set $u':=u\ox R'$ and $\mc R^\dg := \Im(\Sym(u'))$.

Then $\mc R^\dg$ is a homogeneous quotient of $\mc R(N')$, because
$u'\:N'\to F'$ is one of the maps from $N'$ to a free module involved in
the definition of $\mc R(N')$.  Since $u^*\:F^*\to N^*$ is surjective,
so is $u^*\ox R'_f$.  However, forming the dual commutes with flat base
change.  So $u^*\ox R'_f$ is the dual of $u\ox R'_f$.  So
\cite[Prp.\,1.3]{EHU} yields $\mc R(N\ox_R R'_f)= \Im(\Sym(u\ox_RR'_f))$.
 Note that $N\ox_RR'_f=N'\ox_{R'}R'_f$ and that
$u\ox_RR'_f=u'\ox_{R'}R'_f$.  But $R'_f$ is flat over $R'$.  Hence $\mc
R(N\ox_R R'_f)= \mc R(N')\ox_{R'}R'_f$ by \cite[Prp.\,1.3]{EHU}, and
$\Im(\Sym(u\ox_RR'_f))= \Im(\Sym(u'))\ox_{R'}R'_f$.  Thus $\smash{\mc
R(N')_f\risom\mc R^\dg_f}$.

However, $R'_f$ is free over $R_f$, and $R_f$ is flat over $R$; so
$R'_f$ is flat over $R$.  So $\mc R(N\ox_R R'_f)= \mc R(N)\ox_RR'_f$ by
\cite[Prp.\,1.3]{EHU}.  Thus $\smash{\mc R(N)\ox R'_f\risom \mc
R^\dg_f}$.

Moreover, as $u^*\:F^*\to N^*$ is surjective, \cite[Prp.\,1.3]{EHU}
yields $\Im(\Sym(u))=\mc R(N)$ too.  Let $\vf\:\mc R(N)\into \Sym(F)$
denote the inclusion.  Then $\Sym(u')$ factors into a surjection
$\Sym(N')\onto R(N)\ox R'$ followed by $\vf\ox R'\:\mc R(N)\ox R'\to
\Sym(F')$.  So $\Im(\Sym(u'))=\Im(\vf\ox R')$.  But $\mc R^\dg :=
\Im(\Sym(u'))$.  Also $\vf\ox R'_f$ is injective as $R'_f$ is flat over
$R$.  Thus there's a canonical surjection $\mc R(N)\ox R'\onto \mc
R^\dg$, which induces an isomorphism $\smash{\mc R(N)\ox R'_f\risom \mc
R^\dg_f}$.

So $P^\dg$ is a closed subscheme of $P\x_XX'$, and $P^\dg_f= P\x_XX'_f$.
Furthermore, $Z^\dg=(Z\x_XX')\cap P^\dg$.  Thus $B^\dg$ is a closed
subscheme of $B\x_XX'$, and $D^\dg=D'\cap B^\dg$.  Moreover, $X'_f$ is
flat over $X$, and blowup commutes with flat base change.  Thus
$B_f\x_{X_f}X'_f= B^\dg_f$.

Let $\pi\:B^\dg\to B$ be the composition of the inclusion $B^\dg\into
B\x_XX'$ followed by the projection $B\x_XX'\to B$.  As $X'\to X$ is
finite, so is $\pi$.  Further, $\smash{\pi^{-1}B_f=B^\dg_f}$ and
$B^\dg_f$ is dense in $B^\dg$.  Also, $X_f$ is dense in $X$, and
projection sets up a bijection from the components of $B$ to those of
$P$, so to those of $X$; hence, $B_f$ is dense in $B$.  Finally, $R'_f$
is a free $R_f$-module of pure rank $\dl$; hence, $(\pi_*\mc
O_{B^\dg})|B_f$ is a free $\ms O_{B_f}$-module of pure rank $\dl$.  Thus
$\pi$ is of pure degree $\dl$.

Plainly $D^\dg= \pi^{-1}D$.  Thus Lem.\,\ref{lePF}  yields
$\pi_*\bigl(D^\dg \cdot [B^\dg]_s\bigr)=\dl\,D\cdot[B]_s$.

Plainly $\mu^\dg = \pi^*\mu$ and $\nu^\dg = \pi^*\nu$.  So the
projection formula for Chern classes, discussed on p.\,53 of Fulton's
book \cite{Fu} yields, for any $i,\,j\ge0$,
 $$\pi_*\bigl({\mu^\dg}^i{\nu^\dg}^j\cap(D^\dg \cdot [B^\dg]_s)\bigr)
 = \mu^i\nu^j\cap\pi_*\bigl(D^\dg \cdot [B^\dg]_s\bigr).$$
 Hence $\pi_*\Sigma^\dg = \dl\,\Sigma$.  So $\int_{D^\dg/Y}\Sigma^\dg =
\dl\,\int_{D/Y}\Sigma= \dl\,e(M,\,N)$.  But we proved
$\Sigma^\dg=\Sigma'$ on $D'$ and $\int_{D'/Y}\Sigma'=\sum_{\go
m\in\Phi}\dl_\go m\,e(M'_\go m,\,N'_\go m)$.  Thus
Eqn.\,\eqref{eqMain} is proved.
 \end{proof}

\begin{remark}\label{reSthy}
 In Complex Analytic Singularity Theory, $N$ arises as a submodule of a
free module $F$, and $N$ is generically nonzero.  Moreover, usually, $R$
is reduced.  Then automatically, $N$ is torsion-free and generically
free of positive rank.  Thus by \cite[Thm.\,1.6]{EHU}, $\mc R(N)$
is just the $R$-subalgebra $\mc R$ of $\Sym(F)$ generated by $N$.

So $\mc R(N)$ has no $R$-torsion, Also the canonical surjection
$\Sym(N)\onto \mc R(N)$ is generically bijective as $N$ is generically
free of positive rank.  Thus $\mc R(N)$ is equal to the quotient of
$\Sym(N)$ by its $R$-torsion.

More generally, $R$ may be just generically reduced.  So still $N$ is
generically free of positive rank.  Although $N$ need not be
torsion-free, nevertheless the inclusion $N\into F$ factors through the
canonical map $N\to N^{**}$, so the latter is injective; hence, by
\cite[Prp.\,1.8]{EHU}, the kernel of the canonical surjection $\mc
R(N)\onto \mc R$ is nilpotent.  Also this surjection is generically an
isomorphism by the preceding argument.  Thus owing to the argument in
the proof of Thm.\,\ref{thMain} used to compare $\mc R(N')$ and $\mc
R^\dg$, we may use $\mc R$ in place of $\mc R(N)$ to express $e(M,N)$ as
an intersection number.

Given a module-finite $R$-algebra $R'$ of pure degree, the map $N\ox
R'\to F\ox R'$ needn't be injective, although it's generically so.  Let
$N''$ be the image of $N$, and $M''$ the submodule generated by $M$.
Then $\mc R(N\ox R')\to \mc R(N'')$ is surjective by \cite[p.\,702,
mid.]{EHU}, and generically an isomophism.  Thus, the same argument in
the proof of Thm.\,\ref{thMain} yields $e(M''_\go m,\,N''_\go m) =
e(M'_\go m,\,N'_\go m)$ for all $\go m$ in \eqref{eqMain}.
 \end{remark}

\begin{proposition}\label{prAF}
 Let $R$ be a Noetherian local ring of positive dimension, $M\subset N$
nested finitely generated modules with $M/N$ of finite length. For each
minimal prime $\go p$ of $R$, assume $N_\go p$ is a free $R_\go
p$-module of positive rank $r_\go p$, set $N(\go p):=N/\go pN$, let
$M(\go p)$ be the image of $M$ in $N(\go p)$, and set $d_\go
p:=\dim(R/\go p)$ and $\ell_\go p:= \length(R_\go p)$.  Set $s:=s(N)$,
and let $\Lambda$ be the set of\/ $\go p$ with $d_\go p+r_\go p-1=s$.
Then
 \begin{equation}\label{eqAF}
 \ts e(M,\,N)=\sum_{\go p\in\Lambda}\ell_\go p\,e(M(\go p),\,N(\go p)).
 \end{equation}
 \end{proposition}

\begin{proof}
 Start as in the proof of Thm.\,\ref{thMain}.  By definition, $[B]_s:=
\sum_jm_j[B_j]$ where the $B_j$ are the $s$-dimensional components of
$B$ with their reduced structure and $m_j$ is the length of the local
ring of $B_j$ at its generic point.  Recall that projection sets up a
bijection from the components of $B$ to those of $P$ and from those of
$P$ to those of $X$.  For each $j$, let's identify the image of $B_j$ in
$X$ and the value of $m_j$.

Given a minimal prime $\go P$ of $\mc R(N)$, set $\go p:=\go P\cap R$.
Then $\go p$ is a minimal prime of $R$.  Let $\go n$ be the maximal
ideal of $R$.  Then there's $g\in \go n-\go p$ such that $N_g$ is a free
$R_g$-module of rank $r_\go p$.  Hence $P_g = \Proj(\Sym(N_g))$.  Hence
the length of the local ring of $P_g$ at the (generic) point
corresponding to $\go P$ is just $\ell_\go p$.  But $B_g\risom P_g$ as
$g\notin\go n$.  Thus each $B_j$ maps onto a component of $X$ defined by
a minimal prime $\go p$ of $R$, and $\go p\in \Lambda$ because $\dim B_j =
d_\go p+r_\go p-1$; also, the correspondence $B_j\mapsto \go p$, from
$\{B_j\}$ to $\Lambda$, is bijective, and $m_j=\ell_\go p$.

Given $\go p\in\Lambda$, let $B_\go p$ stand for the corresponding $B_j$.
Then Eqn.\,\eqref{eqDfnBR} yields
 $$\ts e(M,\,N) = \sum_{\go p\in\Lambda}\ell_\go p\,
   \int_{D/Y}\sum_{i=1}^s\mu^{i-1}\nu^{s-i}\cap(D\cdot[B_\go p]).$$
 Thus it remains to prove the following formula for each $\go p\in
\Lambda$:
 \begin{equation}\label{eqep}\ts
 \int_{D/Y}\sum_{i=1}^s\mu^{i-1}\nu^{s-i}\cap(D\cdot[B_\go p])
  =e(M(\go p),\,N(\go p)).
 \end{equation}

Fix any minimal prime $\go P$ of $\mc R(N)$, and set $\go p:=\go P\cap
R$.  Let's prove that $\mc R(N(\go p))\risom \mc R(N)/\go P$.  As
before, take $u\:N\to F$ with $F$ free of finite rank and $\Im(\Sym(u))=
\mc R(N)$.  Then $\go p\mc R(N)\subset \go P= \mc R(N)\cap \go
p\Sym(F)$.  Since forming a symmetric algebra commutes with base change,
$\Sym(N)/\go p\Sym(N)= \Sym(N(\go p))$ and $\Sym(F)/\go
p\Sym(F)=\Sym(F/\go pF)$.  Hence $\Sym(u)$ induces the factorization
 $$\Sym(N(\go p))\onto \mc R(N)/\go p\mc R(N)
        \onto \mc R(N)/\go P \into\Sym(F/\go pF).$$
 Moreover, as $N_g$ is free, the two surjections become bijective on
localizing at $g$.  Thus there's a canonical surjection $\mc R(N(\go p))
\onto \mc R(N)/\go P$, and it becomes bijective on localizing at $g$.

 Take $v\:N(\go p)\to G$ with $G$ free over $R/\go p$ and $\Im(\Sym(v))=
\mc R(N(\go p))$.  As $R/\go p$ is a domain, so is $\Sym(G)$; hence, so
is $\mc R(N(\go p))$.  Thus $\mc R(N(\go p)) \risom \mc R(N)/\go P$.

The multiplicity $e(M(\go p),\,N(\go p))$ is defined using a setup
analogous to the one defining $e(M,\,N)$.  Denote the corresponding
objects by $B(\go p)$, $D(\go p)$, etc.  Then $D(\go p)= D\cap B(\go p)$
as $\mc R(N(\go p))\risom \mc R(N)/\go P$.  Also $\mu(\go p)$ and
$\nu(\go p)$ are the pullbacks of $\mu$ and $\nu$.  Assume $\go p\in
\Lambda$, and let $B_\go p$ be as above.  Then $B(\go p)=B_\go p$.  So
 $$ \mu^i\nu^j\cap(D\cdot[B_\go p])
 = \mu(\go p)^i\nu(\go p)^j\cap(D(\go p)\cdot[B(\go p)])
 \tqn{for any} i,\,j\ge0.$$
 Thus Eqn.\,\eqref{eqep} holds, as desired
  \end{proof}

\begin{remark}\label{reZS}
  On pp.\,297--299 in \cite{ZS}, Zariski and Samuel proved a special
case of Thm.\,\ref{thMain} for the ordinary multiplicity.  They
introduced their work with the following statement, which fits the
present paper strikingly well:

\smallskip
{\leftskip=3\parindent\noindent\llap{``}\ignorespaces
We conclude this section with the proof of a theorem which not only can
be used in certain cases for the computation of multiplicities, but also
gives information on the behavior of multiplicities under finite integral
extensions.  This theorem is the algebraic counterpart of the projection
formula for intersection cycles in Algebraic Geometry.''
\par\smallskip}

 To recover their result, take $N:=R$ and take $M$ to be an ideal $\go
q$ that is primary for the maximal ideal, so  $N/M$ is of finite
length.  Then as explained below, our $e(M,\,N)$ becomes their $e(\go
q)$.

They assumed that $R'$ is a module-finite overring of $R$ and that no
nonzero element of $R$ is a zerodivisor in $R'$.  So $R$ is a domain,
and every minimal prime of $R'$ contracts to 0.  So there's a nonzero
$f\in R$ with $R'_f$ a free $R_f$-module of rank $\dl$ where $\dl$ is
the dimension of the total ring of fractions of $R'$ considered as a
vector space over the fraction field of $R$.  Thus $R'$ is of finite
degree $\dl$ over $R$, and our Eqn.\,\eqref{eqMain} recovers their
Eqn.\,(8) on p.\,299.

They didn't exclude the case where $\dim R=0$, as we do, but this case
is rather easy to do directly, even for the Buchsbaum--Rim
multiplicity.  We must exclude it, because in it $Z$ and so $D$ are
empty.  On the other hand, this case is covered by the alternative
geometric treatment of the Buchsbaum--Rim multiplicity in \cite{MBRM}.
That treatment is similar, but simpler, and yields a general
mixed-multiplicity formula.  It begins by replacing the blow-up $B$ of
$P$ along $Z$ by the the completed normal cone of $Z$ in $P$.  So it
isn't suited for the study of generic conditions on $X$, such as a
finite map $X'\to X$ of pure degree.

They defined $e(\go q)$ in terms of the Hilbert--Samuel polynomial;
namely, its leading term is $e(\go q)n^s/s\,!$ where $s:=\dim R$.
Although they didn't relate $e(\go q)$ to geometry, nevertheless the
connection is explained in Samuel's book \cite{SM}: on pp.\,81--82, the
projection formula for cycles is derived from its ``counterpart,'' which
is proved on p.\,32 in his book \cite{SL}.  Furthermore, on pp.\,84--85
of \cite{SM}, the associativity of the intersection product for cycles
is derived via reduction to the diagonal from its algebraic
``counterpart,'' which is proved on pp.\,41--42 in \cite{SL}.  However,
Samuel worked with a restricted class of rings $R$.

This algebraic associativity formula was proved earlier by Chevalley in
Thm.\,5 on p.\,26 in \cite{Ch} as part of one of the first rigorous
treatments of geometric intersection theory and the first based on local
algebra.  However, he worked with geometric local rings, and his
approach does not generalize to arbitrary Noetherian local rings.  The
general case of the formula was proved by Lech \cite{Le} starting from
Samuel's definition of multiplicity.

The formula itself is like \eqref{eqAF} with $N:=R$ and $M:=\go q$, but
there are three differences.  First, $\go q$ is generated by a system of
parameters.  Second, $\go p$ ranges over those minimal primes of an
ideal $\go a$ generated by part of the system, and $\dim(R/\go p) +
\dim(R_\go p)= \dim(R)$.  Third, the length $\ell_\go p$ of $R_\go p$ is
replaced by the multiplicity $e(\go bR_\go p,\,R_\go p)$ where $\go b$
is the ideal generated by the remaining parameters.  Conceivably this
formula could be generalized to the Buchsbaum--Rim multiplicity by
extending the theory of mixed multiplicities developed in 
\cite[\S\,9]{GTBR}.

Taking $\go a = 0$ gives the formula recovered by \eqref{eqAF}.  It was
proved by Serre \cite[V.A).2]{Se}, who named it the {\it additivity
formula}.  Moreover, he defined the multiplicity of $\go q$ on a module,
and observed that, given a short exact sequence, the multiplicity of the
central module is the sum of the multiplicities of the extremes.
(Although here $\go q$ need not be generated by a system of parameters,
the added generality is illusory at least if the residue field of $R$ is
infinite, as $\go q$ can be replaced by a submodule generated by a
system of parameters without changing the multiplicities; see
\cite[Ex.\,4.3.5(a)]{Fu}.)  Many later authors have treated modules,
including Eisenbud, Huneke and Swanson, and Kirby and Rees.

For the Buchsbaum--Rim multiplicity, the modules are modules over the
Rees algebra.  Correspondingly, Thm.\,\ref{thMain} and Prp.\,\ref{prAF}
generalize in straightforward fashion: just replace $[P]$ with the
fundamental cycle of the module; see \cite[\S\,5]{GTBR}.  However, the
added generality doesn't seem to justify the added complexity of
exposition.

The first direct treatment of geometric intersection theory, without
reduction to algebraic counterparts, is found in Fulton's monumental
book \cite{Fu}.  At the start of \S\,4.3 on p.\,79 and in Ex.\,4.3.4 on
p.\,81, he defined the (Samuel) multiplicity as the degree of the
exceptional divisor of certain blow-up; thus he turned a formula given
by Ramanujam \cite{Ra} into a definition.  Also in Ex.\,4.3.4, Fulton
derived the additivity formula.  In Ex.\,4.3.7 on p.\,81, he derived the
projection formula.  In Ex.\,7.1.8 on p.\,123, he gave a version of the
general associativity formula.  The last three formulas are stated in
geometric terms, but the first two translate directly into their
algebraic equivalents over geometric local rings.  However, the third
formula involves intersection multiplicities, and is mathematically a
step away from a direct translation of Chevalley's algebraic
associativity formula.

Basically, our approach follows Fulton's.  However, ours is more
focused, and puts in evidence what must be double checked over an
arbitrary Noetherian local ring.  Furthermore, ours is generalized to
handle the Buchsbaum--Rim multiplicity.  Here the key is Formula
\eqref{eqDfnBR}, which is a generalization of Ramanujam's formula.  Much
of our effort is spent in manipulating Rees algebras.

Just as $e(\go q)$ can be expressed as the normalized leading term of
the Hilbert--Samuel polynomial, $e(M,\,N)$ can be expressed as the
normalized leading term of the Buchsbaum--Rim polynomial, whose value at
$n$ is the length of $\mc R_n(N)/\mc R_n(M)$ for $n\gg0$; namely,
Thm.\,(5.7) of \cite{GTBR} asserts that this polynomial exists, and that
its leading term is $e(M,\,N)n^s/s\,!$ where $s:=s(N)$.  This expression
for $e(M,\,N)$ was taken as its definition, in their setups, by
Buchsbaum and Rim at the bottom of p.\,213 in \cite{BR} and by other
authors, including Huneke and Swanson in Dfn.\,16.5.5 on p.\,316 in
\cite{HS} and Kirby and Rees on p.\,256 in \cite{KR}.

Notice that the Buchsbaum--Rim polynomial generalizes the
Hilbert--Samuel polynomial, since by \cite[Thm.\,1.4]{EHU}, the
abstract Rees algebra of an ideal is equal to the direct sum of its
powers; so $e(M,\,N)$ generalizes $e(\go q)$.  Correspondingly, Kirby
and Rees \cite[Thm.\,6.3,\,iii)]{KR} gave a purely algebraic proof of a
version of Prp.\,\ref{prAF}.  Doubtless, Thm.\,\ref{thMain} has a
similar proof, generalizing Zariski and Samuel's proof.
  \end{remark}

\filbreak

\bibliographystyle{amsplain}

\end{document}